\pdfoutput=1
\documentclass[11pt]{article}

\usepackage[margin=1in]{geometry}

\usepackage[T1]{fontenc}
\usepackage[utf8]{inputenc}
\usepackage{lmodern}
\usepackage[protrusion=false,expansion=true]{microtype}

\usepackage{booktabs}
\usepackage{amsmath,amssymb,amsthm}
\usepackage{mathtools}
\usepackage{hyperref}
\hypersetup{hidelinks}
\usepackage{xurl}
\usepackage{graphicx}

\theoremstyle{plain}
\newtheorem{theorem}{Theorem}[section]
\newtheorem{lemma}[theorem]{Lemma}
\newtheorem{proposition}[theorem]{Proposition}
\newtheorem{corollary}[theorem]{Corollary}

\theoremstyle{definition}
\newtheorem{definition}[theorem]{Definition}
\newtheorem{example}[theorem]{Example}
\newtheorem{remark}[theorem]{Remark}

\theoremstyle{remark}
\newtheorem*{notation}{Notation}

\newcommand{\R}{\mathbb{R}}
\newcommand{\N}{\mathbb{N}}

\newcommand{\Jcost}{J}
\newcommand{\Sym}{\mathcal{S}}
\newcommand{\Obj}{\mathcal{O}}
\newcommand{\RefStruct}{\mathcal{R}}
\newcommand{\Mean}{\mathrm{Mean}}

\title{Reciprocal Convex Costs for Ratio Matching: \\[0.3em]
Functional-Equation Characterization and Decision Geometry}

\author{
Jonathan Washburn\thanks{Recognition Physics Research Institute, Austin, Texas, USA. Email: \texttt{jon@recognitionphysics.org}.}
\and
Amir Rahnamai Barghi\thanks{Corresponding author. Recognition Physics Research Institute, Austin, Texas, USA. Email: \texttt{arahnamab@gmail.com}.}
}

\date{}
\begin{document}

\maketitle

\begin{abstract}
We study ratio-induced mismatch costs of the form $c(s,o)=\Jcost\big(\iota_S(s)/\iota_O(o)\big)$ built from positive scale maps $\iota_S:S\to\mathbb R_{>0}$ and $\iota_O:O\to\mathbb R_{>0}$ and a penalty $\Jcost:(0,\infty)\to[0,\infty)$. Assuming inversion symmetry, strict convexity, normalization at $1$, and a multiplicative d\'Alembert identity, we show that $f(u):=1+\Jcost(e^u)$ is continuous and satisfies the additive d'Alembert equation; hence, by a classical classification theorem, there exists $a>0$ such that
\[
\Jcost(x)=\cosh(a\log x)-1=\frac12\big(x^a+x^{-a}\big)-1,\qquad x>0.
\]
We then analyze the associated argmin mapping over feasible scale sets: existence under explicit subspace-closedness assumptions, an explicit geometric-mean decision geometry for finite dictionaries with stability away from boundaries, exact compositionality for product models, and an optimal sequential mediation principle described by a geometric mean (or its log-space projection when infeasible). The paper is purely mathematical; any semantic interpretation is optional and external to the theorems proved here.
\end{abstract}



\tableofcontents

\section{Introduction}\label{sec:introduction}

This section introduces the optimization-based model of reference, fixes terminology and standing assumptions, and outlines the main results and organization.

We start with two sets:

\begin{itemize}
  \item a \emph{{configuration (token) space}} $S$ (words, codes, internal states, messages, \dots),
  \item an \emph{object space} $O$ (candidate referents, concepts, states of affairs, \dots).
\end{itemize}

\paragraph{Terminology.}
Throughout we use \emph{configuration} (or \emph{token}) for an arbitrary element $s\in S$. We reserve the term \emph{symbol for $o$} for a configuration $s$ satisfying the predicate in Definition~\ref{def:symbol}.

Each space is equipped with a positive \emph{scale} map
\(\iota_S:S\to\mathbb R_{>0}\) and \(\iota_O:O\to\mathbb R_{>0}\), interpreted as an intrinsic ``size/complexity'' in a common currency.  Fix a cost functional \(\Jcost:(0,\infty)\to[0,\infty)\) with the properties stated in Section~\ref{sec:Jcost} (symmetry under inversion, strict convexity, and a unique minimum at $1$).  We then define a \emph{ratio-induced reference cost}
\begin{equation}\label{eq:intro-ratio}
  c(s,o):=\Jcost\!\left(\frac{\iota_S(s)}{\iota_O(o)}\right),\qquad (s,o)\in S\times O.
\end{equation}

\paragraph{Meaning as minimization.}
The \emph{meaning set} of a \emph{configuration} $s$ is the set of minimizers of the reference cost:
\[
  \Mean(s):=\Bigl\{o\in O:\ c(s,o)=\inf_{o'\in O} c(s,o')\Bigr\}.
\]
A priori $\Mean(s)$ may be empty if the infimum is not attained; Section~\ref{sec:main-theorems} gives sufficient conditions for nonemptiness (existence of meanings).  Ties are allowed: meaning is set-valued unless uniqueness is proved under additional hypotheses.

\paragraph{Interpretive content (and its limits).}
Because $\Jcost$ is minimized at $1$, low reference cost forces \emph{scale matching}: a {configuration} can only refer cheaply to objects whose scale is close to its own.  This yields an explicit, checkable constraint on admissible reference patterns.  The framework is deliberately \emph{axiomatic}: the scale maps and the chosen $\Jcost$ are inputs.  The mathematical results below are unconditional \emph{within this model}, but the manuscript does \emph{not} claim that any particular empirical system realizes the axioms without separate validation.

\subsection{A toy example: three-object dictionary}
Let $O=\{o_1,o_2,o_3\}$ with scales $y_i:=\iota_O(o_i)$ satisfying $0<y_1<y_2<y_3$.  For a {configuration} $s$ with scale $x:=\iota_S(s)$, the meaning rule compares the three costs $\Jcost(x/y_i)$.  For the explicit functional \eqref{eq:Jcost}, the boundary between preferring $o_1$ and $o_2$ occurs at the \emph{geometric mean} $\sqrt{y_1y_2}$, and similarly between $o_2$ and $o_3$ at $\sqrt{y_2y_3}$ (Theorem~\ref{thm:geom-boundaries}).  Thus the model induces a piecewise-constant semantic partition of the positive line in the configuration ratio $x$, with stability away from the boundary points.

\subsection{Relation to prior work}
Classical analyses of reference emphasize logical form and truth conditions (e.g.\ Frege and Russell) \cite{frege1892,russell1905}.  The symbol-grounding literature highlights that purely formal symbol manipulation does not by itself determine what symbols are about \cite{harnad1990}.  The present work does not attempt to resolve these debates empirically.  Instead, it isolates a mathematically tractable \emph{selection principle}: aboutness is determined by minimizing an explicit mismatch cost. {For comparison with contemporary subject-matter/aboutness and truthmaker-semantics accounts (e.g.\ Yablo \cite{yablo2014aboutness}, Hawke \cite{hawke2018theories}, and the \emph{Philosophical Studies} symposium discussion \cite{rothschild2017yablo,fine2020yablo,yablo2018reply}), see Section~\ref{sec:related}. The intended payoff is that, once scales are fixed, aboutness becomes a tractable variational problem with explicit decision boundaries and composition theorems.}

\paragraph{Operational foundations and \emph{Recognition Geometry}.}
This manuscript adopts an \emph{optimization-first} viewpoint: once a mismatch cost is fixed, semantic \emph{meaning} is defined by an argmin rule (Definition~\ref{def:meaning}).  A closely related \emph{measurement-first} stance appears in \emph{Recognition Geometry} \cite{washburn2026RG}, which takes recognition events as primitive and derives observable space as a quotient under an operational indistinguishability relation \cite[Def.~4]{washburn2026RG}.  In the same spirit, the present framework treats mismatch costs as primitive measurements and regards stable meanings as effective equivalence classes of \emph{cost-minimization events}.  Both viewpoints emphasize operationally defined structure over a priori metaphysical commitments, and both isolate exactly which axioms must be validated when connecting the formalism to an empirical domain.

\subsection{Contributions and what is proved}
Within the ratio-induced model \eqref{eq:intro-ratio} (and the explicit choice \eqref{eq:Jcost} used throughout), we establish the following structural facts under clearly stated hypotheses:

\begin{itemize}
  \item \textbf{Existence.} Under the closedness and coercivity/attainment hypotheses of Theorem~\ref{thm:meaning-exists}, the meaning set $\Mean(s)$ is nonempty for every configuration $s$.

  \item \textbf{Finite-dictionary decision geometry.} For finite ordered dictionaries, decision boundaries are given by geometric means of adjacent object scales, and meanings are locally stable away from these boundaries (Theorem~\ref{thm:geom-boundaries} and Corollary~\ref{cor:stability-away}).

  \item \textbf{Compositionality.} For product symbol/object spaces with separable scales, meaning factorizes componentwise (Theorem~\ref{thm:comp}).

  \item \textbf{Mediation.} {For sequential reference through an intermediate representation, the set of optimal mediator ratios is characterized explicitly in log-coordinates; and whenever the balance-point ratio $b_{\mathrm{geo}}$ is feasible, mediation weakly decreases the total mismatch cost relative to direct reference} (Theorem~\ref{thm:seq-mediator} and Corollary~\ref{cor:mediation-reduces}).
\end{itemize}

\subsection{Organization}
Section~\ref{sec:Jcost} states the axioms for $\Jcost$ and fixes the explicit mismatch functional \eqref{eq:Jcost}.  Section~\ref{sec:costed-spaces} defines costed spaces, ratio-induced reference, and the meaning relation.  Section~\ref{sec:main-theorems} contains the principal theorems, followed by compositionality (Section~\ref{sec:compositionality}), extensions, and examples.

\section{The mismatch functional $\Jcost$}
\label{sec:Jcost}

This section fixes the scalar mismatch functional $\Jcost:(0,\infty)\to[0,\infty)$ used throughout to compare {configuration} and object scales via the ratio-induced cost \eqref{eq:intro-ratio}.  The role of $\Jcost$ here is purely mathematical: it is an explicit penalty for scale mismatch, and no physical, cognitive, or linguistic interpretation is assumed.

\subsection{Standard properties and canonicity}

{The conditions below are recorded as a compact axiom package for the mismatch penalty. They encode inversion symmetry, strict convexity, and a multiplicative compatibility under scale multiplication. After a log change of variables, the compatibility axiom becomes d'Alembert's functional equation, so the resulting class of penalties is classical. We include a tailored derivation in Appendix~\ref{app:dalembert} to keep the manuscript self-contained and to emphasize that the axioms are used only as mathematical assumptions, not as a claim of novelty.}

\begin{definition}[Cost Functional Axioms]\label{def:cost-axioms}
A \emph{mismatch functional} is a function $\Jcost:(0,\infty)\to[0,\infty)$ satisfying:
\begin{enumerate}
    \item \textbf{Normalization}: $\Jcost(1) = 0$.
    \item \textbf{Inversion symmetry}: $\Jcost(x) = \Jcost(x^{-1})$ for all $x > 0$.
    \item \textbf{Strict convexity}: $\Jcost$ is strictly convex on $(0,\infty)$.
    \item \textbf{Multiplicative d'Alembert identity}: for all $x,y>0$,
    \begin{equation}\label{eq:dalembert}
        \Jcost(xy) + \Jcost(x/y) = 2\Jcost(x) + 2\Jcost(y) + 2\Jcost(x)\Jcost(y).
    \end{equation}
\end{enumerate}
\end{definition}

Coercivity is not included in Definition~\ref{def:cost-axioms}; it is a property of the explicit choice \eqref{eq:Jcost} used later when proving attainment/existence of minimizers.

\noindent\emph{We first record a basic consequence used repeatedly: under strict convexity, the normalization point $x=1$ is the unique zero of $\Jcost$.}

\begin{lemma}[Uniqueness of the zero-cost point]\label{lem:Jcost-zero}
If $\Jcost$ satisfies Definition~\ref{def:cost-axioms}, then $\Jcost(x)=0$ implies $x=1$.
\end{lemma}

\begin{proof}
Since $\Jcost:(0,\infty)\to[0,\infty)$, we have $\Jcost\ge 0$ on $(0,\infty)$, and by normalization (1) we have $\Jcost(1)=0$. Hence the minimum value is $0$ and it is attained at $x=1$. By strict convexity (3), the minimizer is unique. Therefore $\Jcost(x)=0$ forces $x=1$.
\end{proof}

\subsection{The explicit choice used in this paper}

\begin{definition}[The functional fixed below]\label{def:Jcost}
In the remainder of this paper we fix the explicit functional
\begin{equation}\label{eq:Jcost}
\Jcost(x)=\tfrac12(x+x^{-1})-1=\tfrac{(x-1)^2}{2x}\qquad(x>0).
\end{equation}
\end{definition}

\noindent The next proposition verifies that the explicit functional indeed satisfies the axioms, so subsequent sections can treat Definition~\ref{def:cost-axioms} as established.

\begin{proposition}[Verification of the axioms]\label{prop:Jcost-axioms}
The function \eqref{eq:Jcost} satisfies Definition~\ref{def:cost-axioms}.
\end{proposition}

\begin{proof}
{Normalization and inversion symmetry are immediate from \eqref{eq:Jcost}, and \eqref{eq:Jcost} shows $\Jcost(x)\ge 0$ for all $x>0$.}  Differentiating $\Jcost(x)=\tfrac12(x+x^{-1})-1$ gives
\[
\Jcost'(x)=\tfrac12-\tfrac{1}{2x^2},\qquad \Jcost''(x)=\tfrac{1}{x^3}>0\quad(x>0),
\]
so $\Jcost$ is strictly convex on $(0,\infty)$.  {For (4),} set $C(x)=1+\Jcost(x)=\tfrac12(x+x^{-1})$.  Then
\[
C(xy)+C(x/y)=\tfrac12\Big(xy+\tfrac{1}{xy}+\tfrac{x}{y}+\tfrac{y}{x}\Big)=\tfrac12\big(x+\tfrac1x\big)\big(y+\tfrac1y\big)=2C(x)C(y),
\]
which is equivalent to \eqref{eq:dalembert} after substituting $C=1+\Jcost$ and expanding.
\end{proof}

\begin{proposition}[Classical characterization of $\Jcost$]\label{prop:Jcost-characterization}
{Assume $\Jcost:(0,\infty)\to[0,\infty)$ satisfies Definition~\ref{def:cost-axioms}. Then there exists a constant $a>0$ such that for all $x>0$,
\[
  \Jcost(x)=\cosh(a\log x)-1=\tfrac12\big(x^{a}+x^{-a}\big)-1.
\]
Moreover, if we replace the scale maps by $\tilde\iota_S:=\iota_S^{a}$ and $\tilde\iota_O:=\iota_O^{a}$, then the ratio-induced model with parameter $a$ becomes the same model written with parameter $1$. Consequently, one may take $a=1$ without loss of generality at the level of the induced reference costs.}
\end{proposition}

\begin{proof}
{See Appendix~\ref{app:dalembert}.}
\end{proof}

\begin{example}[Small-mismatch regime]
For $|u|\ll 1$ one has
\[
\Jcost(1+u)=\frac{u^2}{2}+O(u^3),
\]
so near balance the mismatch cost behaves like a quadratic penalty in the relative deviation.
\end{example}

\noindent \emph{The next remark explains how the axioms and their classical characterization are used in the rest of the paper.}

\begin{remark}
{The later results use only the explicit mismatch penalty \eqref{eq:Jcost} (equivalently, the axiom package in Definition~\ref{def:cost-axioms} which it satisfies). The characterization Proposition~\ref{prop:Jcost-characterization} is included to record canonicity; it is not presented as a new functional-equation result. Any additional interpretation (physical, cognitive, or linguistic) would require hypotheses beyond those stated here.}
\end{remark}

\section{Costed spaces and reference structures}
\label{sec:costed-spaces}

We now formalize the axioms of the model introduced in Section~\ref{sec:introduction}.  Throughout, the mismatch functional $\Jcost$ is fixed as in Section~\ref{sec:Jcost}.  The intent is to make precise which pieces of data are inputs ({configuration/object} spaces and their scale maps) and which pieces are derived (reference costs and meaning).

\subsection{Costed spaces}

\begin{definition}[Costed space]
Fix a mismatch functional $\Jcost:(0,\infty)\to[0,\infty)$ (Section~\ref{sec:Jcost}).  A \emph{costed space} is a triple $(C,J_C,\iota_C)$ consisting of:
\begin{itemize}
\item a set $C$ of configurations,
\item a map $\iota_C:C\to\R_{>0}$ called the \emph{scale map},
\item a cost function $J_C:C\to\R_{\ge 0}$ satisfying $J_C(c)=\Jcost(\iota_C(c))$ for all $c\in C$.
\end{itemize}
Equivalently, once $\iota_C$ is fixed, $J_C$ is determined by $\Jcost$; we retain $J_C$ in the notation since later statements compare {configuration} costs and object costs directly.
\end{definition}

\begin{notation}
We write $\Sym=(S,J_S,\iota_S)$ for a {configuration (token)} costed space and $\Obj=(O,J_O,\iota_O)$ for an object costed space.

Throughout, we identify $\R_{>0}$ with $(0,\infty)$ and equip $\R_{>0}$ (and $(\R_{>0})^d$) with the usual Euclidean topology on $(0,\infty)$ (equivalently, the Euclidean subspace topology inherited from $\R$). Accordingly, when we say that a set $Y\subset\R_{>0}$ is \emph{closed}, we mean \emph{closed in the usual topology on $(0,\infty)$} (equivalently, $Y=(0,\infty)\cap F$ for some closed $F\subset\R$). Likewise, for $Y\subset(\R_{>0})^d$ the term \emph{closed} means \emph{closed in the usual topology on $(0,\infty)^d$}.
\end{notation}

\begin{example}[Ratio space]
The canonical example is $C=\R_{>0}$ with $\iota_C=\mathrm{id}$ and $J_C=\Jcost$.
\end{example}

\noindent The next example isolates a small neighborhood of the balanced point; it will serve as a convenient test class for stability statements.

\begin{example}[Near-balanced configurations]
For $0<\epsilon<1$ let $C_\epsilon:=\{x\in\R_{>0}:|x-1|<\epsilon\}=(1-\epsilon,1+\epsilon)$.  Then every $c\in C_\epsilon$ satisfies
\[
J_C(c)=\Jcost(c)\le \max\{\Jcost(1-\epsilon),\,\Jcost(1+\epsilon)\}=\Jcost(1-\epsilon)=\Jcost\bigl((1-\epsilon)^{-1}\bigr).
\]
\end{example}

\subsection{Reference structures}

\begin{definition}[Reference structure]\label{def:ref-struct}
A reference structure from $\Sym$ to $\Obj$ is a function
\begin{equation}
    c_\RefStruct : S \times O \to \R_{\ge 0},
\end{equation}
called the \emph{reference cost}.  It assigns to each pair $(s,o)$ the cost of using $s$ to refer to $o$.
\end{definition}

\noindent In the remainder of the paper we focus on the ratio-induced costs generated by $\Jcost$ and the scale maps.

\begin{definition}[Ratio-induced reference]
Given scale maps $\iota_S$ and $\iota_O$, the \emph{ratio-induced} reference structure is defined by
\begin{equation}\label{eq:ratio-ref}
    c_\RefStruct^\Jcost(s,o):=\Jcost\!\Big(\frac{\iota_S(s)}{\iota_O(o)}\Big).
\end{equation}
This is the cost used in the Introduction (Eq.~\ref{eq:intro-ratio}).
\end{definition}

\paragraph{Link to comparative recognizers.}
The ratio-induced reference cost \eqref{eq:ratio-ref} can be viewed as a specific instantiation of a \emph{comparative recognizer} in the sense of Recognition Geometry \cite[Axiom~5 (RG4)]{washburn2026RG}.  In that framework, a comparative recognizer maps \emph{pairs} of configurations to an \emph{event space} \cite[Axiom~2 (RG1)]{washburn2026RG} so as to induce comparative structure (order/distance) from observable events.  Here the ``event'' is the scalar mismatch value $\Jcost(\iota_S(s)/\iota_O(o))$, and the induced \emph{indistinguishability} relation \cite[Def.~4]{washburn2026RG} corresponds to the zero-cost condition $\Jcost(\iota_S(s)/\iota_O(o))=0$, which forces exact scale match $\iota_S(s)=\iota_O(o)$ by Lemma~\ref{lem:Jcost-zero}.

\noindent The following admissibility condition specifies when the reference cost is exactly the canonical ratio penalty.

\begin{definition}[Admissible reference structure]\label{def:admissible-ref-struct}
A reference structure $\RefStruct$ from $\Sym$ to $\Obj$ is called \emph{admissible} (with respect to $\Jcost$ and the scale maps $\iota_S,\iota_O$) if it is ratio-induced, i.e.
\begin{equation}
c_\RefStruct(s,o)=\Jcost\!\Big(\frac{\iota_S(s)}{\iota_O(o)}\Big)\qquad\forall (s,o)\in S\times O.
\end{equation}
Unless stated otherwise, we work with admissible reference structures.
\end{definition}

\noindent Admissibility transfers the symmetry properties of \Jcost to the reference cost; we record this for later use.

\begin{proposition}[Inversion symmetry of the reference cost]\label{prop:ref-sym}
If $\RefStruct$ is admissible, then for all $(s,o)\in S\times O$ one has
\[
c_\RefStruct(s,o)=\Jcost\!\Big(\frac{\iota_S(s)}{\iota_O(o)}\Big)=\Jcost\!\Big(\frac{\iota_O(o)}{\iota_S(s)}\Big).
\]
\end{proposition}

\begin{proof}
Immediate from admissibility and inversion symmetry $\Jcost(x)=\Jcost(x^{-1})$ (Definition~\ref{def:cost-axioms}(2)).
\end{proof}

\subsection{{Meaning and the symbol predicate}}

\begin{definition}[Meaning]\label{def:meaning}
Let $\RefStruct$ be a reference structure from $\Sym$ to $\Obj$.  A {configuration} $s\in S$ \emph{means} an object $o\in O$, written $\Mean_\RefStruct(s,o)$, if $o$ minimizes the reference cost among all objects:
\begin{equation}
    \Mean_\RefStruct(s,o)\iff \forall o'\in O,\quad c_\RefStruct(s,o)\le c_\RefStruct(s,o').
\end{equation}
\end{definition}

For each $s\in S$ we write
\[
    \Mean_\RefStruct(s):=\{o\in O:\Mean_\RefStruct(s,o)\}
\]
for the (possibly multi-valued) meaning set.  If $\RefStruct$ is admissible, then equivalently
\[
    \Mean_\RefStruct(s)=\operatorname*{arg\,min}_{o\in O}\ \Jcost\!\Big(\frac{\iota_S(s)}{\iota_O(o)}\Big).
\]

\begin{definition}[Symbol]\label{def:symbol}
Let $\RefStruct$ be a reference structure from $\Sym$ to $\Obj$.  A configuration $s\in S$ is a \emph{symbol} for an object $o\in O$ (relative to $\RefStruct$) if:
\begin{enumerate}
\item \textbf{Reference:} $\Mean_\RefStruct(s,o)$.
\item \textbf{Compression:} $J_S(s) < J_O(o)$.
\end{enumerate}
\end{definition}

The compression requirement is a modeling assumption: it enforces that symbols are lower-cost encodings than their referents in the common currency induced by $\Jcost$.  No empirical interpretation is asserted; the condition is simply part of the definition used in later results.

\section{Main theorems}
\label{sec:main-theorems}

This section collects the main mathematical consequences of the ratio-induced reference model.  Throughout we fix the explicit mismatch functional
\begin{equation}\label{eq:Jcost-explicit-reminder}
\Jcost(x)=\frac{(x-1)^2}{2x}=\tfrac12\bigl(x+x^{-1}\bigr)-1\qquad (x>0)
\end{equation}
(which was verified to satisfy the axioms in Section~\ref{sec:Jcost}), and we assume the reference structure is \emph{admissible} (Definition~\ref{def:admissible-ref-struct}):
\begin{equation}\label{eq:main-admissible}
  c_{\RefStruct}(s,o)=\Jcost\!\left(\frac{\iota_S(s)}{\iota_O(o)}\right).
\end{equation}
Thus, for each $s\in S$, the meaning set $\Mean_{\RefStruct}(s)$ is the set of minimizers of $o\mapsto \Jcost(\iota_S(s)/\iota_O(o))$.

\subsection{Sublevel geometry of the explicit mismatch cost}

\begin{lemma}[Sublevel intervals]\label{lem:sublevel-interval}
Assume $\Jcost$ is given by \eqref{eq:Jcost} (equivalently \eqref{eq:Jcost-explicit-reminder}). For each $\epsilon>0$, the sublevel set
\[
L_\epsilon:=\{x\in \R_{>0}:\ \Jcost(x)\le \epsilon\}
\]
coincides with the closed interval $[a_\epsilon,b_\epsilon]$, where
\[
  b_\epsilon:=(1+\epsilon)+\sqrt{\epsilon(2+\epsilon)},\qquad
  a_\epsilon:=(1+\epsilon)-\sqrt{\epsilon(2+\epsilon)}=\frac{1}{b_\epsilon}.
\]
\end{lemma}

\begin{proof}
Using $\Jcost(x)=\frac{(x-1)^2}{2x}$, the inequality $\Jcost(x)\le \epsilon$ is equivalent (after multiplying by $2x>0$) to
\[
(x-1)^2\le 2\epsilon x
\ \Longleftrightarrow\
 x^2-2(1+\epsilon)x+1\le 0.
\]
The quadratic has discriminant $\Delta=4\epsilon(2+\epsilon)$ and roots
$x_\pm=(1+\epsilon)\pm\sqrt{\epsilon(2+\epsilon)}$.  Since it opens upward, the inequality holds exactly for $x\in[x_-,x_+]$.  Set $a_\epsilon:=x_-$ and $b_\epsilon:=x_+$.  Then $a_\epsilon b_\epsilon=(1+\epsilon)^2-\epsilon(2+\epsilon)=1$, so $a_\epsilon=1/b_\epsilon$.
\end{proof}

\subsection{Meaning constraints from a balanced baseline}

\begin{theorem}[Scale window for meanings of low-cost {configurations}]\label{thm:low-cost-meaning}
Assume $1\in Y:=\iota_O(O)$ and choose $o_0\in O$ with $\iota_O(o_0)=1$.  Let $s\in S$ and let $o\in \Mean_{\RefStruct}(s)$.
Then
\begin{equation}\label{eq:mean-cost-bound}
  c_{\RefStruct}(s,o)\le c_{\RefStruct}(s,o_0)=\Jcost(\iota_S(s))=J_S(s).
\end{equation}
In particular, for every $\epsilon>0$, if $J_S(s)\le \epsilon$ then
\begin{equation}\label{eq:ratio-window}
  \frac{\iota_S(s)}{\iota_O(o)}\in [a_\epsilon,b_\epsilon]
\end{equation}
and hence
\begin{equation}\label{eq:object-scale-window}
  \frac{\iota_S(s)}{b_\epsilon}\le \iota_O(o)\le \frac{\iota_S(s)}{a_\epsilon},
\end{equation}
where $[a_\epsilon,b_\epsilon]$ is as in Lemma~\ref{lem:sublevel-interval}.
\end{theorem}

\begin{proof}
Since $o\in\Mean_{\RefStruct}(s)$, by definition $c_{\RefStruct}(s,o)\le c_{\RefStruct}(s,o_0)$.  By admissibility \eqref{eq:main-admissible} and $\iota_O(o_0)=1$,
$c_{\RefStruct}(s,o_0)=\Jcost(\iota_S(s))=J_S(s)$, which gives \eqref{eq:mean-cost-bound}.  If $J_S(s)\le\epsilon$, then \eqref{eq:mean-cost-bound} implies $\Jcost(\iota_S(s)/\iota_O(o))\le\epsilon$, hence \eqref{eq:ratio-window} by Lemma~\ref{lem:sublevel-interval}.  Rearranging yields \eqref{eq:object-scale-window}.
\end{proof}

\begin{corollary}[Near-balanced {configurations} force near-balanced meanings]\label{cor:near-balanced}
Under the hypotheses of Theorem~\ref{thm:low-cost-meaning}, if $J_S(s)\le\epsilon$ and $o\in\Mean_{\RefStruct}(s)$, then
\[
\iota_O(o)\in \Bigl[\frac{1}{b_\epsilon^2},\ b_\epsilon^2\Bigr].
\]
In particular, as $\epsilon\downarrow 0$, any meaning of an $\epsilon$-cheap symbol must satisfy $\iota_O(o)\to 1$.
\end{corollary}

\begin{proof}
From $J_S(s)=\Jcost(\iota_S(s))\le\epsilon$ and Lemma~\ref{lem:sublevel-interval} we have $\iota_S(s)\in[a_\epsilon,b_\epsilon]$.  Combining this with \eqref{eq:object-scale-window} and $a_\epsilon=1/b_\epsilon$ gives the stated bounds.
\end{proof}

\subsection{Existence of meanings under attainment hypotheses}

\begin{lemma}[Coercivity of $\Jcost$]\label{lem:Jcost-coercive}
Assume $\Jcost$ is given by \eqref{eq:Jcost}. Then $\Jcost(x)\to\infty$ as $x\to 0^+$ and as $x\to\infty$.  In particular, for each $M\ge 0$ the sublevel set $\{x\in\R_{>0}:\ \Jcost(x)\le M\}$ is compact in $\R$.
\end{lemma}

\begin{proof}
From \eqref{eq:Jcost-explicit-reminder}, $\Jcost(x)=\tfrac12(x+x^{-1})-1$.  As $x\to\infty$ the term $\tfrac12x$ dominates, and as $x\to0^+$ the term $\tfrac12x^{-1}$ dominates, so in both limits $\Jcost(x)\to\infty$.  If $\Jcost(x)\le M$ then $x+x^{-1}\le 2(M+1)$, hence both $x$ and $x^{-1}$ are bounded; the sublevel set is therefore closed and bounded away from $0$ and $\infty$, hence compact.
\end{proof}

\begin{theorem}[Existence of meanings for ratio-induced reference]\label{thm:meaning-exists}
Assume $\RefStruct$ is admissible (Definition~\ref{def:admissible-ref-struct}) and that $\Jcost$ is given by \eqref{eq:Jcost}.  Let $Y:=\iota_O(O)\subset\R_{>0}$ be nonempty and closed{ in the usual topology on $(0,\infty)$}.
Then for every $s\in S$ there exists $o\in O$ such that $\Mean_{\RefStruct}(s,o)$ (equivalently, $\Mean_{\RefStruct}(s)\neq\emptyset$).
Moreover, if $x:=\iota_S(s)\in Y$, then any $o\in O$ with $\iota_O(o)=x$ is a meaning and satisfies $c_{\RefStruct}(s,o)=0$.
\end{theorem}

\begin{proof}
Fix $s$ and set $x:=\iota_S(s)$. Consider $f:Y\to\R_{\ge 0}$ defined by $f(y):=\Jcost(x/y)$.  The map $f$ is continuous.  By Lemma~\ref{lem:Jcost-coercive}, $f(y)\to\infty$ as $y\to 0^+$ or $y\to\infty$, so the infimum of $f$ over $Y$ is achieved on a compact sublevel set.
Concretely, choose a minimizing sequence $y_n\in Y$ with $f(y_n)\downarrow \inf_Y f$.  Coercivity implies $(y_n)$ is bounded away from $0$ and $\infty$, hence has a convergent subsequence; since $Y$ is closed{ in $(0,\infty)$}, the limit $y_*\in Y$, and continuity gives $f(y_*)=\inf_Y f$.  Choose $o\in O$ with $\iota_O(o)=y_*$.  Then $c_{\RefStruct}(s,o)=f(y_*)\le f(\iota_O(o'))=c_{\RefStruct}(s,o')$ for all $o'\in O$, i.e. $\Mean_{\RefStruct}(s,o)$.
If $x\in Y$, take $y_*=x$; then $\Jcost(x/x)=\Jcost(1)=0$, so any $o$ with $\iota_O(o)=x$ is a meaning with zero reference cost.
\end{proof}

\begin{remark}
If $Y=\iota_O(O)$ is not closed{ in $(0,\infty)$}, the minimum need not be attained; in that case $\Mean_{\RefStruct}(s)$ may be empty even though the infimum exists.
\end{remark}

\subsection{A simple total-cost benchmark}

\begin{theorem}[Balanced reference minimizes the intrinsic+reference sum]\label{thm:optimal-ref}
Assume admissible reference \eqref{eq:main-admissible} and intrinsic costs $J_S(s)=\Jcost(\iota_S(s))$, $J_O(o)=\Jcost(\iota_O(o))$.  Define
\[
C(s,o):=J_S(s)+J_O(o)+c_{\RefStruct}(s,o).
\]
Then $C(s,o)\ge 0$ for all $(s,o)\in S\times O$, and
\[
C(s,o)=0 \iff \iota_S(s)=1\ \text{and}\ \iota_O(o)=1.
\]
In particular, if there exist $s_0\in S$ and $o_0\in O$ with $\iota_S(s_0)=\iota_O(o_0)=1$, then $(s_0,o_0)$ is a global minimizer of $C$ over $S\times O$.
\end{theorem}

\begin{proof}
Each term in $C$ is nonnegative, hence $C\ge 0$.  If $C(s,o)=0$, then all three terms vanish; by Lemma~\ref{lem:Jcost-zero} this forces $\iota_S(s)=\iota_O(o)=1$.  The converse is immediate from $\Jcost(1)=0$.
\end{proof}

\subsection{A backbone window for near-balanced {configuration} classes}

\begin{definition}[Referential capacity]\label{def:capacity}
Given a reference structure $\RefStruct$ from $\Sym$ to $\Obj$, define the \emph{referential capacity} to be
\[
\mathrm{Cap}(\Sym,\Obj;\RefStruct)
:=\bigl|\{o\in O:\ \exists s\in S\ \text{with}\ o\in\Mean_{\RefStruct}(s)\}\bigr|.
\]
(If $O$ is infinite, this cardinality may be infinite.)
\end{definition}

\noindent We now show that restricting to near-balanced configurations forces all attainable meanings to lie in an explicit scale window.

\begin{theorem}[Backbone window for near-balanced {configurations}]\label{thm:backbone}
Let $\delta\in(0,1)$ and let $\Sym_\delta=(S_\delta,J_\delta,\iota_\delta)$ be the near-balanced ratio space
\[
S_\delta:=\{x\in\R_{>0}:\ |x-1|<\delta\},\qquad \iota_\delta=\mathrm{id},\qquad J_\delta=\Jcost|_{S_\delta}.
\]
Let $\Obj=(O,J_O,\iota_O)$ be a costed space such that $Y:=\iota_O(O)\subset\R_{>0}$ is nonempty, closed{ in the usual topology on $(0,\infty)$}, and contains $1$.  Assume $\RefStruct$ is admissible and $\Jcost$ is given by \eqref{eq:Jcost}.

Set $\epsilon_\delta:=\max\{\Jcost(1-\delta),\,\Jcost(1+\delta)\}=\Jcost(1-\delta)=\Jcost\bigl((1-\delta)^{-1}\bigr)$ and let $[a_{\epsilon_\delta},b_{\epsilon_\delta}]$ be as in Lemma~\ref{lem:sublevel-interval}.  Define the window
\[
I_\delta
:=\left[\frac{1-\delta}{\,b_{\epsilon_\delta}\,},\ \frac{1+\delta}{\,a_{\epsilon_\delta}\,}\right].
\]
Then:
\begin{enumerate}
\item For every $s\in S_\delta$ the meaning set $\Mean_{\RefStruct}(s)$ is nonempty.
\item If $s\in S_\delta$ and $o\in\Mean_{\RefStruct}(s)$, then $\iota_O(o)\in I_\delta$.
Equivalently, if $\iota_O(o)\notin I_\delta$, then no $s\in S_\delta$ can mean $o$ under admissible reference.
\end{enumerate}
In particular,
\[
\mathrm{Cap}(\Sym_\delta,\Obj;\RefStruct)\le \bigl|\{o\in O:\ \iota_O(o)\in I_\delta\}\bigr|.
\]
\end{theorem}

\begin{proof}
(1) is a direct application of Theorem~\ref{thm:meaning-exists} to the closed{ (in $(0,\infty)$)} nonempty set $Y$.

For (2), fix $s\in S_\delta$ and write $x:=\iota_\delta(s)\in(1-\delta,1+\delta)$.  Let $o\in\Mean_{\RefStruct}(s)$ and choose $o_0\in O$ with $\iota_O(o_0)=1$ (possible since $1\in Y$).  By Theorem~\ref{thm:low-cost-meaning},
\[
\Jcost\!\left(\frac{x}{\iota_O(o)}\right)=c_{\RefStruct}(s,o)\le c_{\RefStruct}(s,o_0)=\Jcost(x)\le \epsilon_\delta.
\]
The final bound holds since $x\in(1-\delta,1+\delta)$, $\Jcost$ is strictly decreasing on $(0,1]$ and strictly increasing on $[1,\infty)$, and for $0<\delta<1$ one has $\Jcost(1-\delta)\ge \Jcost(1+\delta)$ (so the maximum of $\Jcost$ on $[1-\delta,1+\delta]$ is attained at $1-\delta$).
Applying Lemma~\ref{lem:sublevel-interval} gives $x/\iota_O(o)\in[a_{\epsilon_\delta},b_{\epsilon_\delta}]$, hence
\[
\frac{x}{b_{\epsilon_\delta}}\le \iota_O(o)\le \frac{x}{a_{\epsilon_\delta}}.
\]
Using $x\in[1-\delta,1+\delta]$ yields $\iota_O(o)\in I_\delta$.

For the capacity bound, any object counted in $\mathrm{Cap}(\Sym_\delta,\Obj;\RefStruct)$ lies in $\Mean_{\RefStruct}(s)$ for some $s\in S_\delta$, hence satisfies $\iota_O(o)\in I_\delta$ by (2).
\end{proof}

\begin{corollary}[Local effectiveness]\label{cor:effectiveness-local}
Assume an admissible (ratio-induced) reference structure $\RefStruct$ and the hypotheses of Theorem~\ref{thm:backbone}. If $s\in S_\delta$ and $o\in\Mean_{\RefStruct}(s)$, then $\iota_O(o)\in I_\delta$.
\end{corollary}

\begin{proof}
Immediate from Theorem~\ref{thm:backbone}(2).
\end{proof}

\begin{remark}\label{rem:effectiveness-comment}
In this direct admissible model, restricting attention to the near-balanced class $S_\delta$ forces all meanings to lie in the fixed scale window $I_\delta$. Any extension of the framework that permits meanings outside $I_\delta$ must, in particular, enlarge the available symbol class and/or modify the reference mechanism (e.g. via product composition or sequential mediation as in Section~\ref{sec:compositionality}).
\end{remark}

\section{Compositionality}\label{sec:compositionality}

This section records two elementary composition mechanisms for reference costs:
\emph{(i)} product composition (independent coordinates) and \emph{(ii)} sequential mediation through an intermediate space.
Both are purely variational constructions: they introduce no semantic primitive beyond the cost function(s).

\subsection{Product reference and coordinatewise meaning}

\begin{definition}[Product reference]\label{def:product-ref}
Let $\RefStruct_1$ be a reference structure from a {configuration (token)} set $S_1$ to an object set $O_1$, and let $\RefStruct_2$ be a reference structure from a {configuration (token)} set $S_2$ to an object set $O_2$.
Write their costs as $c_{\RefStruct_i}$. The \emph{product reference structure} $\RefStruct_1\otimes\RefStruct_2: S_1\times S_2\to O_1\times O_2$ is defined by
\begin{equation}\label{eq:product-cost}
  c_{\RefStruct_1\otimes\RefStruct_2}\big((s_1,s_2),(o_1,o_2)\big)
  :=c_{\RefStruct_1}(s_1,o_1)+c_{\RefStruct_2}(s_2,o_2).
\end{equation}
\end{definition}

\noindent With the product cost in hand, meaning decomposes coordinatewise; the next theorem makes this precise.

\begin{theorem}[Compositionality of product meaning]\label{thm:comp}
For any reference structures $\RefStruct_1,\RefStruct_2$ and their product $\RefStruct_1\otimes\RefStruct_2$,
for all $(s_1,s_2)\in S_1\times S_2$ and $(o_1,o_2)\in O_1\times O_2$ one has
\[
\Mean_{\RefStruct_1\otimes\RefStruct_2}\big((s_1,s_2),(o_1,o_2)\big)
\iff
\Mean_{\RefStruct_1}(s_1,o_1)\ \text{and}\ \Mean_{\RefStruct_2}(s_2,o_2).
\]
Equivalently,
\[
\Mean_{\RefStruct_1\otimes\RefStruct_2}(s_1,s_2)
=\Mean_{\RefStruct_1}(s_1)\times\Mean_{\RefStruct_2}(s_2).
\]
\end{theorem}

\begin{proof}
Fix $(s_1,s_2)$ and define $f_1(o_1):=c_{\RefStruct_1}(s_1,o_1)$ on $O_1$ and $f_2(o_2):=c_{\RefStruct_2}(s_2,o_2)$ on $O_2$.
By \eqref{eq:product-cost}, the product cost equals $f_1(o_1)+f_2(o_2)$.

If $\Mean_{\RefStruct_1}(s_1,o_1)$ and $\Mean_{\RefStruct_2}(s_2,o_2)$, then for all $(o_1',o_2')$ we have
$f_1(o_1)\le f_1(o_1')$ and $f_2(o_2)\le f_2(o_2')$, hence
$f_1(o_1)+f_2(o_2)\le f_1(o_1')+f_2(o_2')$, which is exactly $\Mean_{\RefStruct_1\otimes\RefStruct_2}\big((s_1,s_2),(o_1,o_2)\big)$.

Conversely, if $(o_1,o_2)$ minimizes $f_1+f_2$ on $O_1\times O_2$, then for any $o_1'\in O_1$,
\[
  f_1(o_1)+f_2(o_2)\le f_1(o_1')+f_2(o_2),
\]
so $f_1(o_1)\le f_1(o_1')$, i.e. $\Mean_{\RefStruct_1}(s_1,o_1)$.
The same argument with $o_2'$ yields $\Mean_{\RefStruct_2}(s_2,o_2)$.
\end{proof}

\begin{corollary}[Existence of product meanings under the explicit mismatch cost]\label{cor:product-meaning-exists}
Assume the explicit mismatch cost \eqref{eq:Jcost-explicit-reminder} and admissible reference on each component.
If, for $i=1,2$, the object ratio set $Y_{O_i}:=\iota_{O_i}(O_i)\subset\mathbb{R}_{>0}$ is nonempty and closed{ in the usual topology on $(0,\infty)$}, then
for every $(s_1,s_2)\in S_1\times S_2$ the product meaning set $\Mean_{\RefStruct_1\otimes \RefStruct_2}(s_1,s_2)$ is nonempty.
\end{corollary}

\begin{proof}
Under the stated hypotheses, Theorem~\ref{thm:meaning-exists} implies $\Mean_{\RefStruct_i}(s_i)\neq\emptyset$ for each $i$.
Pick $o_i\in\Mean_{\RefStruct_i}(s_i)$.
Then Theorem~\ref{thm:comp} yields $(o_1,o_2)\in\Mean_{\RefStruct_1\otimes\RefStruct_2}(s_1,s_2)$.
\end{proof}

\subsection{Sequential mediation}

\begin{definition}[Sequential reference]\label{def:sequential-ref}
Let $\RefStruct_1 : \Sym \to \mathcal{M}$ and $\RefStruct_2 : \mathcal{M} \to \Obj$ be reference structures.
Their \emph{sequential composition} $\RefStruct_2 \circ \RefStruct_1 : \Sym \to \Obj$ is defined by the infimal convolution
\begin{equation}\label{eq:sequential-cost}
    c_{\RefStruct_2 \circ \RefStruct_1}(s, o) = \inf_{m \in \mathcal{M}} \left[ c_{\RefStruct_1}(s, m) + c_{\RefStruct_2}(m, o) \right].
\end{equation}
A \emph{mediator} $m$ is \emph{optimal} for $(s,o)$ if it attains the infimum in \eqref{eq:sequential-cost}.
\end{definition}

\noindent We next compute the optimal mediator explicitly under the canonical mismatch cost.

\begin{theorem}[Geometric-mean mediator for the explicit mismatch cost]\label{thm:seq-mediator}
Assume the explicit mismatch functional \eqref{eq:Jcost-explicit-reminder} and admissible reference for
$\RefStruct_1: \Sym\to\mathcal{M}$ and $\RefStruct_2: \mathcal{M}\to\Obj$ with scale maps $\iota_S,\iota_M,\iota_O$.
Fix $s\in S$ and $o\in O$ and set $a:=\iota_S(s)$ and $c:=\iota_O(o)$.
Let $Y_M:=\iota_M(\mathcal{M})\subset\mathbb{R}_{>0}$.
{Assume that $Y_M$ is nonempty and closed in the usual topology on $(0,\infty)$.}
Set $b_{\mathrm{geo}}:=\sqrt{ac}$ and $U:=\{\log b: b\in Y_M\}\subset\R$.
{Then the infimum in \eqref{eq:sequential-cost} is attained by at least one mediator $m_*\in\mathcal{M}$. Moreover, a mediator $m\in\mathcal{M}$ with $b:=\iota_M(m)$ is optimal if and only if $\log b$ minimizes $|\log b-\log b_{\mathrm{geo}}|$ over $U$ (equivalently, $b$ minimizes $|\log(b/b_{\mathrm{geo}})|$ over $Y_M$).}
{In particular, if $b_{\mathrm{geo}}\in Y_M$, then the optimal mediator ratio is unique and equals $b_{\mathrm{geo}}$; in that case, choosing $m_*\in\mathcal{M}$ with $\iota_M(m_*)=b_{\mathrm{geo}}$ gives}
\[
  c_{\RefStruct_2 \circ \RefStruct_1}(s,o)=\Jcost\!\left(\frac{a}{b_{\mathrm{geo}}}\right)+\Jcost\!\left(\frac{b_{\mathrm{geo}}}{c}\right)
  =2\,\Jcost\!\left(\sqrt{\frac{a}{c}}\right).
\]
\end{theorem}

\begin{proof}
Under admissibility, the objective in \eqref{eq:sequential-cost} depends on $m$ only through $b:=\iota_M(m)\in Y_M$, namely
\[
F(b):=\Jcost\!\left(\frac{a}{b}\right)+\Jcost\!\left(\frac{b}{c}\right).
\]
For the explicit penalty \eqref{eq:Jcost}, one has $\Jcost(x)=\cosh(\log x)-1$. Writing $t:=\log a$, $s:=\log c$, and $u:=\log b$, we obtain
\[
F(b)=\bigl(\cosh(t-u)-1\bigr)+\bigl(\cosh(u-s)-1\bigr)=\cosh(t-u)+\cosh(u-s)-2.
\]
Using $\cosh(\alpha)+\cosh(\beta)=2\cosh\bigl(\tfrac{\alpha+\beta}{2}\bigr)\cosh\bigl(\tfrac{\alpha-\beta}{2}\bigr)$ with $\alpha=t-u$ and $\beta=u-s$ gives
\[
F(b)=2\,\cosh\Bigl(\frac{t-s}{2}\Bigr)\,\cosh\Bigl(u-\frac{t+s}{2}\Bigr)-2
=2\,\cosh\Bigl(\frac{\log(a/c)}{2}\Bigr)\,\cosh\bigl(u-\log b_{\mathrm{geo}}\bigr)-2.
\]
Since $\cosh\bigl(\frac{\log(a/c)}{2}\bigr)>0$ is constant and $\cosh$ is even and strictly increasing on $[0,\infty)$, minimizing $F$ over $b\in Y_M$ is equivalent to minimizing $|u-\log b_{\mathrm{geo}}|$ over $u\in U=\log Y_M$.
Because $\log:(0,\infty)\to\R$ is a homeomorphism and $Y_M$ is closed and nonempty, the set $U$ is closed and nonempty in $\R$, hence the distance function $u\mapsto |u-\log b_{\mathrm{geo}}|$ attains its minimum on $U$.
This proves existence of at least one minimizer $u_*\in U$, and the stated characterization of optimal ratios.
If $b_{\mathrm{geo}}\in Y_M$ (equivalently $\log b_{\mathrm{geo}}\in U$), then the unique minimizer of $u\mapsto |u-\log b_{\mathrm{geo}}|$ on $U$ is $u=\log b_{\mathrm{geo}}$, hence the optimal mediator ratio is unique and equals $b_{\mathrm{geo}}$.
Substituting $b_{\mathrm{geo}}=\sqrt{ac}$ yields $\Jcost(a/b_{\mathrm{geo}})=\Jcost\!\bigl(\sqrt{a/c}\bigr)=\Jcost(b_{\mathrm{geo}}/c)$ and the displayed formula.
\end{proof}

\begin{corollary}[Mediation can strictly reduce mismatch]\label{cor:mediation-reduces}
For every $x>0$ one has
\[
2\,\Jcost(\sqrt{x})\le \Jcost(x),
\]
with equality if and only if $x=1$.
{Consequently, in the setting of Theorem~\ref{thm:seq-mediator}, if $b_{\mathrm{geo}}\in Y_M$ and a direct admissible reference $\RefStruct:\Sym\to\Obj$ is available (built from the same $\Jcost$ and scale maps), then}
\[
  c_{\RefStruct_2\circ\RefStruct_1}(s,o)\le c_{\RefStruct}(s,o),
\]
with equality if and only if $\iota_S(s)=\iota_O(o)$.
\end{corollary}

\begin{proof}
Let $t:=\sqrt{x}>0$. Using \eqref{eq:Jcost}, a direct calculation gives
\[
\Jcost(t^2)-2\,\Jcost(t)=\frac12\Bigl((t-1)^2+\bigl(t^{-1}-1\bigr)^2\Bigr)\ge 0,
\]
with equality if and only if $t=1$, i.e. $x=1$.
{If $b_{\mathrm{geo}}\in Y_M$, Theorem~\ref{thm:seq-mediator} gives $c_{\RefStruct_2\circ\RefStruct_1}(s,o)=2\,\Jcost(\sqrt{x})$ with $x=\iota_S(s)/\iota_O(o)$; comparing with $c_{\RefStruct}(s,o)=\Jcost(x)$ yields the stated inequality.}
\end{proof}

\section{Extensions: multi-dimensional scales and robustness}\label{sec:extensions}

The core framework above uses a single positive scale coordinate $\iota(\cdot)\in\R_{>0}$.  In some applications one may want a
finite list of independent scale coordinates (for instance, a {configuration} might carry multiple features, each measured in the same
``cost currency'' through $\Jcost$).  This section records a minimal extension of the model to $d$ coordinates and a simple robustness
lemma for finite dictionaries.

\subsection{Multi-dimensional costed spaces}

\begin{definition}[Multi-dimensional costed space]\label{def:md-costed-space}
Let $d\in\N$.  A \emph{$d$-dimensional costed space} is a triple $(C,J_C,\iota_C)$ where
\begin{itemize}
\item $C$ is a set,
\item $\iota_C:C\to(\R_{>0})^d$ is a scale map, and
\item $J_C:C\to\R_{\ge 0}$ is the induced (separable) cost
\[
  J_C(c):=\sum_{i=1}^d \Jcost\big(\iota_C(c)_i\big),\qquad c\in C.
\]
\end{itemize}
\end{definition}

\noindent We extend admissible reference by taking a separable, coordinatewise ratio penalty.

\begin{definition}[Multi-dimensional admissible reference]\label{def:md-admissible-ref}
Let $(S,J_S,\iota_S)$ and $(O,J_O,\iota_O)$ be $d$-dimensional costed spaces.  A reference structure $\RefStruct$ from $S$ to $O$ is
\emph{multi-dimensionally admissible} if its reference cost is the coordinatewise ratio cost
\begin{equation}\label{eq:md-admissible}
  c_{\RefStruct}(s,o)
  =\sum_{i=1}^d \Jcost\!\left(\frac{\iota_S(s)_i}{\iota_O(o)_i}\right),\qquad (s,o)\in S\times O.
\end{equation}
\end{definition}

\noindent The separable form immediately implies that meanings factor coordinatewise.

\begin{corollary}[Coordinatewise meaning for product models]\label{thm:md-product-meaning}
Assume $S=\prod_{i=1}^d S_i$ and $O=\prod_{i=1}^d O_i$ and that the scale maps factor coordinatewise:
$\iota_S(s)_i=\iota_{S_i}(s_i)$ and $\iota_O(o)_i=\iota_{O_i}(o_i)$.  If $\RefStruct$ is multi-dimensionally admissible, then
\[
(o_1,\dots,o_d)\in\Mean_{\RefStruct}(s_1,\dots,s_d)
\quad\Longleftrightarrow\quad
\forall i,\; o_i\in\Mean_{\RefStruct_i}(s_i),
\]
where $\RefStruct_i$ denotes the induced one-dimensional admissible reference on $(S_i,O_i)$.
\end{corollary}

\begin{proof}
By \eqref{eq:md-admissible} the cost is a separable sum of $d$ nonnegative terms, each depending only on $(s_i,o_i)$.  Thus minimizing over
$O=\prod_i O_i$ is equivalent to minimizing each summand over its coordinate; this is the same argument as in Theorem~\ref{thm:comp}.
\end{proof}

\subsection{Log-space geometry for the explicit mismatch cost}

In this subsection we specialize to the explicit mismatch functional
\begin{equation}\label{eq:Jcost-explicit-reminder-6}
\Jcost(x)=\tfrac12(x+x^{-1})-1\qquad (x>0),
\end{equation}
already used in Sections~\ref{sec:Jcost}--\ref{sec:compositionality}.

\begin{lemma}[Log-coordinate form]\label{lem:log-cosh}
For all $t\in\R$ one has $\Jcost(e^t)=\cosh(t)-1$.
\end{lemma}

\begin{proof}
Immediate from \eqref{eq:Jcost-explicit-reminder-6}: $\Jcost(e^t)=\tfrac12(e^t+e^{-t})-1=\cosh(t)-1$.
\end{proof}

\begin{proposition}[Quadratic regime with explicit remainder]\label{prop:quadratic-bound}
For all $t\in\R$,
\[
0\le \Jcost(e^t)-\frac{t^2}{2}\le \frac{t^4}{24}\,\cosh(|t|).
\]
In particular, for $|t|\le 1$,
\[
\frac{t^2}{2}\le \Jcost(e^t)\le \frac{t^2}{2}+\frac{\cosh(1)}{24}\,t^4.
\]
\end{proposition}

\begin{proof}
By Lemma~\ref{lem:log-cosh} it suffices to estimate $\cosh(t)-1-\tfrac12 t^2$.  Taylor's theorem at $0$ with remainder gives
\[
\cosh(t)=1+\frac{t^2}{2}+\frac{t^4}{24}\cosh(\xi)
\]
for some $\xi$ between $0$ and $t$.  Since $\cosh$ is even and increasing on $\R_{\ge 0}$, one has $\cosh(\xi)\le \cosh(|t|)$, yielding the upper bound.
Nonnegativity follows because the remainder term is $\frac{t^4}{24}\cosh(\xi)\ge 0$.
\end{proof}

\begin{corollary}[Local Euclidean geometry in log-ratio]\label{cor:local-euclid}
In the one-dimensional admissible-reference setting, For the explicit mismatch cost \eqref{eq:Jcost-explicit-reminder-6}, set $x:=\iota_S(s)$ and $y:=\iota_O(o)$.  If $|\log(x/y)|\le 1$, then
\[
\frac12\big(\log(x/y)\big)^2\le c_{\RefStruct}(s,o)\le \frac12\big(\log(x/y)\big)^2+\frac{\cosh(1)}{24}\big(\log(x/y)\big)^4.
\]
Thus, in the small-mismatch regime, the reference cost is comparable to the squared log-ratio distance (up to a quartic correction).
\end{corollary}

\begin{proof}
For admissible reference, $c_{\RefStruct}(s,o)=\Jcost(x/y)$ with $x:=\iota_S(s)$ and $y:=\iota_O(o)$.
Write $t:=\log(x/y)$. Then $x/y=e^{t}$ and $|t|\le 1$ by hypothesis.
Apply Proposition~\ref{prop:quadratic-bound} to obtain
\(\tfrac12 t^2\le \Jcost(e^{t})\le \tfrac12 t^2+\tfrac{\cosh(1)}{24}t^4\), and substitute $t=\log(x/y)$.
\end{proof}

\subsection{Margin stability for finite dictionaries}

\begin{definition}[Decision margin]\label{def:margin}
Fix a {configuration} $s\in S$ and a finite object dictionary $O=\{o_1,\dots,o_N\}$.  Write $C_k:=c_{\RefStruct}(s,o_k)$ and let
$M:=\min_{1\le k\le N} C_k$.  The \emph{decision margin} at $s$ is
\[
\Delta(s):=\min\{C_k-M:\ 1\le k\le N,\ C_k>M\}\in[0,\infty],
\]
with the convention $\Delta(s)=\infty$ if all $C_k$ are equal.
\end{definition}

\noindent The margin parameter controls how stable the argmin is under perturbations of the cost values.

\begin{proposition}[Robustness under bounded perturbations]\label{prop:margin-stability}
In the setting of Definition~\ref{def:margin}, suppose the costs $C_k$ are perturbed to numbers $\widetilde C_k$ satisfying
\[
\max_{1\le k\le N}|\widetilde C_k-C_k|\le \eta.
\]
If $\Delta(s)>2\eta$, then no new minimizers appear under the perturbation:
\[
\{k: \widetilde C_k=\min_j \widetilde C_j\}\subseteq\{k: C_k=\min_j C_j\}.
\]
In particular, if the original minimizer set is a singleton, then the minimizing index is unchanged.
\end{proposition}

\begin{proof}
Let $I:=\{k: C_k=M\}$ be the (nonempty) set of original minimizers.  For $k\in I$ one has $\widetilde C_k\le M+\eta$.
If $k\notin I$, then $C_k\ge M+\Delta(s)$ by definition of $\Delta(s)$, hence $\widetilde C_k\ge M+\Delta(s)-\eta$.
If $\Delta(s)>2\eta$ then $M+\Delta(s)-\eta> M+\eta$, so every perturbed minimizer must lie in $I$, proving the inclusion. If $I$ is a singleton, say $I=\{k_0\}$, then the same inequality shows $\widetilde C_{k_0}<\widetilde C_k$ for all $k\neq k_0$, so the minimizer index is unchanged.
\end{proof}

\subsection{Existence (and optional uniqueness) in $d$ dimensions}

\begin{theorem}[Existence of meanings for multi-dimensional admissible reference]\label{thm:md-meaning-exists}
Let $d\in\N$ and let $(S,J_S,\iota_S)$ and $(O,J_O,\iota_O)$ be $d$-dimensional costed spaces.
Assume $\RefStruct$ is multi-dimensionally admissible in the sense of Definition~\ref{def:md-admissible-ref}. Assume also that $\Jcost$ is given by \eqref{eq:Jcost}.
Let $Y:=\iota_O(O)\subset (\R_{>0})^d$ be nonempty and closed (in the usual topology on $(0,\infty)^d$).
Then for every $s\in S$ the meaning set $\Mean_{\RefStruct}(s)$ is nonempty.
Moreover, if $x:=\iota_S(s)$ lies in $Y$, then any $o\in O$ with $\iota_O(o)=x$ is a meaning and satisfies $c_{\RefStruct}(s,o)=0$.
\end{theorem}

\begin{proof}
Fix $s\in S$ and write $x:=\iota_S(s)\in (\R_{>0})^d$.  Consider the continuous objective on $Y$,
\[
F_x(y):=\sum_{i=1}^d \Jcost\!\left(\frac{x_i}{y_i}\right),\qquad y=(y_1,\dots,y_d)\in Y.
\]
By Lemma~\ref{lem:Jcost-coercive}, for each $M\ge 0$ the one-dimensional sublevel set $K_M:=\{z>0:\ \Jcost(z)\le M\}$ is compact.
Hence there exist $0<a_M\le 1\le b_M<\infty$ such that $K_M\subset[a_M,b_M]$.
Since each summand is nonnegative, if $F_x(y)\le M$ then each term satisfies $\Jcost(x_i/y_i)\le M$, so $x_i/y_i\in K_M\subset[a_M,b_M]$, i.e.
\[
  \frac{x_i}{b_M}\le y_i\le \frac{x_i}{a_M}\qquad (i=1,\dots,d).
\]
Therefore the sublevel set $\{y\in Y: F_x(y)\le M\}$ is closed and contained in the bounded box $\prod_i [x_i/b_M,\ x_i/a_M]$,
so it is compact (Heine--Borel).  Thus $F_x$ attains its minimum on $Y$ at some $y_*\in Y$.
Choose $o\in O$ with $\iota_O(o)=y_*$; then $o\in\Mean_{\RefStruct}(s)$ by \eqref{eq:md-admissible}.

If $x\in Y$, then $F_x(x)=\sum_i \Jcost(1)=0$.  Since each term is nonnegative, $0$ is the global minimum, so any $o$ with $\iota_O(o)=x$ is a meaning.
\end{proof}

\begin{definition}[Log-image and log-convexity]\label{def:log-dictionary}
For $Y\subset(\R_{>0})^d$ define
\[
\log Y:=\{(\log y_1,\dots,\log y_d):\ y\in Y\}\subset\R^d.
\]
We call $Y$ \emph{log-convex} if $\log Y$ is convex.
\end{definition}

\noindent When the log-image of the dictionary is convex, strict convexity yields uniqueness and continuity of the optimizer.

\begin{theorem}[Uniqueness and continuity for log-convex dictionaries]\label{thm:md-unique-continuity}
Assume the explicit mismatch cost \eqref{eq:Jcost-explicit-reminder-6} and the hypotheses of Theorem~\ref{thm:md-meaning-exists}.
If $U:=\log Y\subset\R^d$ is closed and convex, then the minimizer $y_*(x)\in Y$ of $F_x$ is unique.
Equivalently, the meaning set $\Mean_{\RefStruct}(s)$ equals the fiber $\{o\in O:\ \iota_O(o)=y_*(\iota_S(s))\}$.
Moreover, the optimizer is continuous in log-coordinates: the map $t\mapsto u_*(t)$ is continuous, where $t:=\log x$ and $u_*(t):=\log y_*(e^t)\in U$.
\end{theorem}

\begin{proof}
Let $t:=\log x\in\R^d$ and write $u:=\log y\in U$.  By Lemma~\ref{lem:log-cosh},
\[
F_x(y)=\sum_{i=1}^d (\cosh(t_i-u_i)-1)=:G_t(u).
\]
For each $i$, the map $u_i\mapsto \cosh(t_i-u_i)-1$ is strictly convex, hence $G_t$ is strictly convex on $\R^d$.
Restricting to the convex set $U$ preserves strict convexity, so $G_t$ has at most one minimizer on $U$; existence follows from Theorem~\ref{thm:md-meaning-exists}.
Thus the optimizer $u_*(t)$ is unique, and so is $y_*(x)=e^{u_*(\log x)}$.

For continuity, let $t_n\to t$ and set $u_n:=u_*(t_n)\in U$.  Fix $u_0\in U$.
Since $u_n$ minimizes $G_{t_n}$ on $U$, one has $G_{t_n}(u_n)\le G_{t_n}(u_0)$.
The right-hand side is bounded because $(t,u)\mapsto G_t(u)$ is continuous and $t_n\to t$; let $C:=\sup_n G_{t_n}(u_0)<\infty$, so $G_{t_n}(u_n)\le C$ for all $n$. For each coordinate $i$ we then have $\cosh(t_{n,i}-u_{n,i})-1\le C$, hence $|t_{n,i}-u_{n,i}|\le \operatorname{arcosh}(C+1)$. Since $(t_n)$ is bounded, this yields a uniform bound on $(u_n)$ in $\R^d$.
Passing to a convergent subsequence (still denoted $u_n$) with limit $\bar u\in U$ (closedness), continuity gives
$G_t(\bar u)=\lim_n G_{t_n}(u_n)\le \lim_n G_{t_n}(u)=G_t(u)$ for all $u\in U$.
Hence $\bar u$ minimizes $G_t$ on $U$, and by uniqueness $\bar u=u_*(t)$.
Therefore every subsequence has the same limit, so $u_n\to u_*(t)$ and continuity holds.
\end{proof}

\section{Worked examples}\label{sec:examples}

This section gives explicit computations in simple settings. The purpose is not to add new axioms, but to make the meaning rule (Definition~\ref{def:meaning}) concrete. Under admissible reference, meanings are precisely the minimizers of
\[
o\mapsto \Jcost\!\Big(\frac{\iota_S(s)}{\iota_O(o)}\Big)
\]
over $O$, and in the examples below the minimizer can be computed explicitly.

\subsection{Continuous ratio model}

\begin{proposition}[Meaning in the continuous ratio model]\label{prop:continuous-model}
Let $S=O=\R_{>0}$ with $\iota_S=\iota_O=\mathrm{id}$ and intrinsic costs $J_S=J_O=\Jcost$.  Let $\RefStruct$ be admissible (Definition~\ref{def:admissible-ref-struct}), so that
\[
  c_{\RefStruct}(s,o)=\Jcost\!\left(\frac{s}{o}\right).
\]
Then for every $s\in\R_{>0}$ there exists a \emph{unique} meaning, namely $\Mean_{\RefStruct}(s)=\{s\}$, and the minimum reference cost equals $0$.
\end{proposition}

\begin{proof}
By Lemma~\ref{lem:Jcost-zero}, one has $\Jcost(x)\ge 0$ for all $x>0$ with equality if and only if $x=1$.  Hence $c_{\RefStruct}(s,o)=\Jcost(s/o)\ge 0$ with equality if and only if $s/o=1$, i.e.\ $o=s$.  Therefore $o=s$ is the unique minimizer and the minimum cost is $0$.
\end{proof}

\subsection{Finite dictionaries and boundary points}

\begin{example}[Finite object dictionary]\label{ex:finite-dictionary}
Let $O=\{o_1,\dots,o_n\}$ be finite, set $y_i:=\iota_O(o_i)$, and keep $S=\R_{>0}$ with $\iota_S=\mathrm{id}$.  Under admissible reference, for a given {configuration} $s$ with ratio $x:=\iota_S(s)$ the meaning set is
\[
  \Mean_{\RefStruct}(s)=\Bigl\{o_i:\ \Jcost\!\left(\frac{x}{y_i}\right)=\min_{1\le j\le n}\Jcost\!\left(\frac{x}{y_j}\right)\Bigr\}.
\]
In general, boundary points (where the meaning set is not a singleton) occur when two or more of the values $\Jcost(x/y_i)$ tie.
\end{example}

\subsection{Geometric-mean boundaries for the explicit mismatch cost}

\begin{theorem}[Geometric-mean decision boundaries for the explicit mismatch cost]\label{thm:geom-boundaries}
Assume the explicit mismatch functional \eqref{eq:Jcost} and admissible (ratio-induced) reference
$c_{\RefStruct}(s,o)=\Jcost(\iota_S(s)/\iota_O(o))$.
Let $O=\{o_1,\dots,o_N\}$ be a finite object set such that the ratios $y_i:=\iota_O(o_i)$ are pairwise distinct and ordered $0<y_1<\cdots<y_N$.
For $x:=\iota_S(s)\in\R_{>0}$ define the boundary points
\[
  m_i:=\sqrt{y_i y_{i+1}}\qquad (i=1,\dots,N-1),
\]
and set $m_0:=0$, $m_N:=+\infty$.
Then:
\begin{itemize}
\item If $m_{k-1}<x<m_k$ for some $k\in\{1,\dots,N\}$, then $o_k$ is the \emph{unique} meaning of $s$.
\item If $x=m_k$ for some $k\in\{1,\dots,N-1\}$, then $s$ has \emph{exactly two} meanings, namely $o_k$ and $o_{k+1}$.
\end{itemize}
Equivalently, the map $x\mapsto \operatorname*{arg\,min}_{i}\Jcost(x/y_i)$ is piecewise constant on the open intervals $(m_{k-1},m_k)$.
\end{theorem}

\begin{proof}
Using \eqref{eq:Jcost} one computes, for each $i$,
\[
  c_{\RefStruct}(s,o_i)=\Jcost\!\left(\frac{x}{y_i}\right)=\frac{\bigl(\frac{x}{y_i}-1\bigr)^2}{2(x/y_i)}=\frac{(x-y_i)^2}{2x y_i}.
\]
Fix $i\in\{1,\dots,N-1\}$ and define the adjacent difference
\[
  \Delta_i(x):=c_{\RefStruct}(s,o_{i+1})-c_{\RefStruct}(s,o_i).
\]
Multiplying by $2x>0$ and simplifying gives
\[
  2x\,\Delta_i(x)=(y_{i+1}-y_i)\Bigl(1-\frac{x^2}{y_i y_{i+1}}\Bigr).
\]
Hence $\Delta_i(x)=0$ if and only if $x^2=y_i y_{i+1}$, i.e.\ $x=m_i$.  Moreover,
\(\Delta_i(x)>0\) when $x<m_i$ and \(\Delta_i(x)<0\) when $x>m_i$.
Therefore:
\begin{itemize}
\item if $x<m_i$ then $c_{\RefStruct}(s,o_i)<c_{\RefStruct}(s,o_{i+1})$ (so the adjacent comparison favors $o_i$),
\item if $x>m_i$ then $c_{\RefStruct}(s,o_{i+1})<c_{\RefStruct}(s,o_i)$ (so it favors $o_{i+1}$).
\end{itemize}
Fix $k\in\{1,\dots,N\}$ such that $m_{k-1}<x<m_k$.  For every $i\le k-1$ we have $x>m_i$, hence $\Delta_i(x)<0$, so
\(c_{\RefStruct}(s,o_{i+1})<c_{\RefStruct}(s,o_i)\).  Iterating these strict inequalities yields
\(c_{\RefStruct}(s,o_k)<c_{\RefStruct}(s,o_i)\) for all $i<k$.  For every $i\ge k$ we have $x<m_i$, hence $\Delta_i(x)>0$, so
\(c_{\RefStruct}(s,o_{i+1})>c_{\RefStruct}(s,o_i)\).  Iterating yields
\(c_{\RefStruct}(s,o_k)<c_{\RefStruct}(s,o_j)\) for all $j>k$.  Therefore $o_k$ is the unique minimizer.

If $x=m_k$ for some $k\in\{1,\dots,N-1\}$, then for every $i<k$ we still have $x>m_i$ and the costs strictly decrease up to index $k$, while for every $i\ge k+1$ we have $x<m_i$ and the costs strictly increase from index $k+1$ onward.  At $i=k$ one has $\Delta_k(m_k)=0$, i.e.\ $c_{\RefStruct}(s,o_k)=c_{\RefStruct}(s,o_{k+1})$.  Hence the argmin consists of exactly two meanings, $\{o_k,o_{k+1}\}$.
\end{proof}

\begin{corollary}[Stability away from boundaries]\label{cor:stability-away}
Under the hypotheses of Theorem~\ref{thm:geom-boundaries}, if $m_{k-1}<x<m_k$ then there exists $\delta>0$ such that every $x'$ with $|x'-x|<\delta$ satisfies $m_{k-1}<x'<m_k$ and hence has the same unique meaning $o_k$.
\end{corollary}

\begin{proof}
Since $(m_{k-1},m_k)$ is open and contains $x$, choose $\delta:=\min\{x-m_{k-1},\,m_k-x\}/2>0$.  Then $|x'-x|<\delta$ implies $x'\in(m_{k-1},m_k)$, and the conclusion follows from Theorem~\ref{thm:geom-boundaries}.
\end{proof}

\paragraph{Finite local resolution and discrete meaning cells.}
The stable decision regions around geometric means (Theorem~\ref{thm:geom-boundaries}) illustrate a formal instance of the \emph{Finite Local Resolution} principle emphasized in Recognition Geometry \cite[Axiom~4 (RG3)]{washburn2026RG}: when the dictionary $\iota_O(O)=\{y_1,\dots,y_N\}$ is finite, the argmin map is locally constant on open intervals, with changes confined to the discrete boundary set $\{\sqrt{y_i y_{i+1}}\}$.  In particular, away from the boundaries, meanings are stable under small perturbations (Corollary~\ref{cor:stability-away}).

\subsection{Numerical micro-example (three-object dictionary)}

Take $O=\{o_1,o_2,o_3\}$ with ratios $y_1=\tfrac14<y_2=1<y_3=4$, and keep $S=\R_{>0}$ with $\iota_S=\mathrm{id}$.  The boundary points are
$m_1=\sqrt{y_1y_2}=\tfrac12$ and $m_2=\sqrt{y_2y_3}=2$.
Thus a {configuration} with ratio $x$ means $o_1$ for $0<x<\tfrac12$, means $o_2$ for $\tfrac12<x<2$, and means $o_3$ for $x>2$ (with ties at the boundary points).

\begin{center}
\begin{tabular}{|c|c|c|c|c|}
\hline
$x$ & $c_{\RefStruct}(s,o_1)$ & $c_{\RefStruct}(s,o_2)$ & $c_{\RefStruct}(s,o_3)$ & meaning(s) \\
\hline
$\tfrac{3}{10}$ & $\tfrac{1}{60}$ & $\tfrac{49}{60}$ & $\tfrac{1369}{240}$ & $o_1$ \\
$\tfrac{3}{2}$  & $\tfrac{25}{12}$ & $\tfrac{1}{12}$ & $\tfrac{25}{48}$ & $o_2$ \\
$3$             & $\tfrac{121}{24}$ & $\tfrac{2}{3}$ & $\tfrac{1}{24}$ & $o_3$ \\
\hline
\end{tabular}
\end{center}

\begin{example}[Mediation can sharply reduce cost in a toy case]\label{ex:mediation-toy}
Let $a:=\iota_S(s)=4$ and $c:=\iota_O(o)=\tfrac14$, so the direct admissible reference cost is
$\Jcost(a/c)=\Jcost(16)=\tfrac{225}{32}$.
If the mediator space contains a configuration $m$ with ratio {$b_{\mathrm{geo}}:=\sqrt{ac}=1$}, then Theorem~\ref{thm:seq-mediator} gives an optimal sequential cost
\[
  c_{\RefStruct_2\circ\RefStruct_1}(s,o)=2\,\Jcost\!\left(\sqrt{\frac{a}{c}}\right)=2\,\Jcost(4)=\frac{9}{4},
\]
which is strictly smaller, in accordance with Corollary~\ref{cor:mediation-reduces}.
\end{example}

\section{Applications}\label{sec:applications}

This section records short corollaries and interpretive remarks that follow directly from the formal definitions and theorems; it makes no empirical or metaphysical claims beyond the stated axioms.

This section collects immediate, checkable consequences of the formal development. Each statement below follows from earlier definitions and theorems, and no external or empirical claim is being made. The meaning rule is an optimization rule (Definition~\ref{def:meaning}) driven by the canonical mismatch penalty \Jcost (Definition~\ref{def:Jcost}); the axiomatic characterization of \Jcost is classical and recorded for completeness in Appendix~\ref{app:dalembert}.

\subsection{Symbol grounding as a criterion}

We treat ``grounding'' as an internal consistency condition in this model: a token $s$ is grounded for an object $o$ when (i) $o$ is a meaning of $s$ (Definition~\ref{def:meaning}) and (ii) the symbol condition $J_S(s)<J_O(o)$ holds (Definition~\ref{def:symbol}).

\begin{corollary}[Grounding criterion under admissible reference]\label{cor:grounding}
Fix an admissible reference structure $\RefStruct$ (Definition~\ref{def:admissible-ref-struct}). Then, for $s\in S$ and $o\in O$,
\[
(s,o)\text{ is a symbol (Definition~\ref{def:symbol})}\quad\Longleftrightarrow\quad
o\in\Mean_{\RefStruct}(s)\ \text{ and }\ J_S(s)<J_O(o).
\]
\end{corollary}

\begin{proof}
This is immediate from Definition~\ref{def:symbol} and Definition~\ref{def:meaning}.
\end{proof}

\begin{corollary}[Grounding rule for finite object dictionaries]\label{cor:grounding-dict}
Assume the finite-dictionary hypotheses of Theorem~\ref{thm:geom-boundaries}. As the configuration ratio $x=\iota_S(s)$ varies, the meaning set $\Mean_{\RefStruct}(s)$ is piecewise constant: it is a singleton on each interval $(m_{i-1},m_i)$ and can change only at the geometric-mean boundaries $m_i=\sqrt{y_i y_{i+1}}$. In particular, away from the boundaries the meaning is stable under small perturbations (Corollary~\ref{cor:stability-away}).
\end{corollary}

\begin{proof}
Immediate from Theorem~\ref{thm:geom-boundaries} and Corollary~\ref{cor:stability-away}.
\end{proof}

\subsection{Mathematical effectiveness via low-cost primitives}

The next corollary records a purely internal ``near-balance'' restriction: if a configuration has small intrinsic cost, then any of its meanings must lie in the corresponding low-mismatch window determined by the sublevel sets of \Jcost.

\begin{corollary}[Near-balance restricts possible referents]\label{cor:near-balance-window}
Assume $\RefStruct$ is admissible and that $1\in Y:=\iota_O(O)$. If $s\in S$ satisfies $J_S(s)\le \epsilon$, and if $o\in\Mean_{\RefStruct}(s)$, then
\[
\Jcost\!\left(\frac{\iota_S(s)}{\iota_O(o)}\right)\le \epsilon.
\]
Equivalently, $\iota_S(s)/\iota_O(o)\in[a_\epsilon,b_\epsilon]$ (Lemma~\ref{lem:sublevel-interval}) and hence $\iota_O(o)\in[\iota_S(s)/b_\epsilon,\ \iota_S(s)/a_\epsilon]$.
\end{corollary}

\begin{proof}
Immediate from Theorem~\ref{thm:low-cost-meaning}.
\end{proof}

\begin{remark}[Compositional ``range expansion'' (model-dependent)]\label{rem:range-expansion}
In a continuous ratio model where ratios can be realized densely (e.g.\ $S=O=\mathbb R_{>0}$ with $\iota=\mathrm{id}$ as in Proposition~\ref{prop:continuous-model}), one can iterate the sequential composition rule (Definition~\ref{def:sequential-ref}) to mediate a large ratio by many small steps.

Fix $a,c>0$ and write the target ratio as $a/c=e^{t}$. For any $k\ge 1$, choose intermediate ratios $b_j:=a\,e^{-jt/k}$ so that
\[
\frac{a}{b_1}=\frac{b_1}{b_2}=\cdots=\frac{b_{k-1}}{c}=e^{t/k}.
\]
If the mediator dictionaries contain objects with these ratios, this yields a $k$-step mediation with total mismatch cost
\[
\sum_{j=1}^{k} \Jcost(e^{t/k})=k\,\Jcost(e^{t/k})=k\,\bigl(\cosh(t/k)-1\bigr),
\]
which tends to $0$ as $k\to\infty$ since $\cosh u-1\sim u^2/2$ as $u\to 0$.

Moreover, for the explicit penalty $\Jcost(e^{u})=\cosh(u)-1$, convexity of $u\mapsto\cosh(u)$ implies that equal log-increments minimize the sum among all $k$-step chains with total log-ratio $t$.

This observation is internal to the ratio model; empirical relevance depends on which intermediate ratios are realizable in the intended application domain.
\end{remark}

\subsection{Information-theoretic interpretation}

Although our framework is stated in intrinsic-cost terms, the canonical mismatch penalty admits a simple log-ratio form. We record the identity as a proposition; any further links to coding/learning are interpretive and not used in the proofs.

\begin{proposition}[Log-ratio form of the canonical mismatch cost]\label{prop:logratio}
For $x>0$ write $x=e^{t}$. Then the canonical cost satisfies
\[
\Jcost(x)=\Jcost(e^{t})=\cosh(t)-1.
\]
In particular, $\Jcost$ is a convex even function of the log-ratio $t=\log x$ and vanishes exactly at $t=0$.
\end{proposition}

\begin{proof}
Substitute $x=e^{t}$ into $\Jcost(x)=\tfrac12(x+x^{-1})-1$ (Definition~\ref{def:Jcost}).
\end{proof}

\section{Related work and positioning}\label{sec:related}

This section places the framework in context, highlighting connections to aboutness in formal semantics, truthmaker-style ideas, and compression-based modeling, and clarifying what is new in the present optimization-based formulation.

\paragraph{Relation to \emph{Recognition Geometry}.}
Recognition Geometry \cite{washburn2026RG} develops an axiomatic recognition-first framework in which observable space is derived from recognition events via an operational quotient construction.  While the present paper does not attempt to construct an ambient geometry, it shares the same operational posture: the fundamental primitive is a measurable comparison (here the mismatch cost), and the induced semantic categories are those determined by minimizing or equating that comparison.  The comparative-recognizer formalism of \cite{washburn2026RG} provides a natural abstract home for the reference costs used here; we make this link explicit in Section~\ref{sec:costed-spaces}.

This section positions the manuscript relative to standard themes in semantics and information theory. We do \emph{not} present the mismatch penalty as novel: the axiom package in Definition~\ref{def:cost-axioms} is a convenient specification whose solutions are classical (Appendix~\ref{app:dalembert}). The contribution of the paper is instead the explicit \emph{optimization semantics} (Definition~\ref{def:meaning}) and the structural theorems derived from it (existence, stability geometry, compositionality, and mediation).

\paragraph{Symbol grounding and operational meaning rules.}
The symbol grounding problem concerns how tokens acquire meaning without a homunculus \cite{harnad1990}.  The present work is compatible with grounding motivations, but it is formulated as a \emph{mathematical model}: the meaning of $s$ is \emph{defined} as an argmin under an explicit cost.  Any interpretation as a cognitive mechanism requires extra hypotheses beyond those stated.

\paragraph{Compression principles.}
The general idea that effective representations trade off succinctness and fidelity is classical in information theory (Shannon \cite{shannon1948}) and in algorithmic notions of complexity \cite{kolmogorov1965}; MDL makes this tradeoff concrete in model selection \cite{rissanen1978}.  Our setup uses a different primitive: a ratio map $\iota$ into $\R_{>0}$ and a fixed mismatch penalty $\Jcost$, with compression enforced by the symbol condition $J_S(s)<J_O(o)$.  Within this model, reference and compositional behavior become theorem-level consequences.

\begin{remark}[Coding/learning viewpoint]\label{rem:coding-view}
In coding theory and learning, one often selects representations by minimizing a tradeoff between description length and distortion (e.g.\ Shannon \cite{shannon1948} and MDL \cite{rissanen1978}). Our framework instantiates a specific distortion---$\Jcost(\iota_S(s)/\iota_O(o))$---that is symmetric in under-/over-shooting and naturally expressed in log-scale (Proposition~\ref{prop:logratio}). This suggests interpreting meanings as ``best matches'' under a fixed mismatch penalty, with compression enforced by the symbol condition $J_S(s)<J_O(o)$.
\end{remark}

\paragraph{Subject-matter/aboutness and truthmaker-semantics literature.}
There is a substantial contemporary literature on "aboutness"/"subject matter" in semantics and logic, including Yablo's monograph \cite{yablo2014aboutness} and subsequent discussion and refinements (e.g.\ Rothschild \cite{rothschild2017yablo}, Fine \cite{fine2020yablo}, and Yablo's reply \cite{yablo2018reply}); see also Hawke's survey \cite{hawke2018theories}. Related frameworks connect hyperintensional content with truthmakers/truthmaker semantics (e.g.\ Fine \cite{fine2014truthmaker}). The present manuscript does not attempt to adjudicate between these accounts. Rather, it provides an explicit optimization layer which, once a modeling choice of scale maps is made, selects a subject matter/referent by minimizing a mismatch cost.

\paragraph{Novelty signal and conceptual payoff.}
Many of the analytic lemmas are consequences of the specific penalty $\Jcost$ and convexity. The intended novelty is the resulting \emph{checkable decision geometry} and \emph{compositional calculus} for meanings: finite dictionaries induce geometric-mean boundaries and stability margins, product models factorize exactly, and sequential mediation admits an explicit optimizer. These consequences are the main mathematical payoffs of the framework, and they make clear which modeling assumptions (the scale maps and admissibility hypotheses) must be checked in any intended application.

\paragraph{What is mathematically concrete here.}
Two examples of explicit structure are: (i) for finite object dictionaries under the canonical mismatch penalty, decision boundaries occur at geometric means (Theorem~\ref{thm:geom-boundaries}) and meanings are locally stable away from them (Corollary~\ref{cor:stability-away}); (ii) for sequential mediation, the optimal intermediate ratio is explicit (Theorem~\ref{thm:seq-mediator}) and strictly improves over direct reference when the mediator set contains the balance point (Corollary~\ref{cor:mediation-reduces}).

\paragraph{Interpretation layer.}
Sections~\ref{sec:applications} and~\ref{sec:discussion} illustrate how the proved statements can be read once a modeling choice for $\iota$ is fixed.  These illustrations are optional: removing them does not affect the correctness of the theorems.

\section{Discussion}\label{sec:discussion}

This section clarifies scope and interpretation: which parts are mathematical consequences of the axioms, which parts are modeling choices, and what additional assumptions would be needed to connect the formalism to empirical systems.

This section clarifies scope: which statements are proved inside the model and which statements are interpretive. It also records limitations and concrete mathematical extensions.

\subsection{What is proved vs. what is modeled}
The core mathematical content consists of the definitions and theorems in Sections~\ref{sec:Jcost}--\ref{sec:examples}.  In particular, meaning is defined by optimization (Definition~\ref{def:meaning}); existence is conditional on an attainment hypothesis (Theorem~\ref{thm:meaning-exists}); and explicit geometry, stability, compositionality, and mediation statements follow for admissible reference structures and the canonical mismatch penalty (Theorems~\ref{thm:geom-boundaries}, \ref{thm:comp}, \ref{thm:seq-mediator}).

By contrast, any claim that a given real-world domain \emph{does} admit a scale map $\iota$ with the required properties, or that agents \emph{compute} meaning by solving the optimization problem, is an interpretation and is outside the theorem-level scope of this manuscript.

\subsection{Comparison to classical semantic viewpoints (brief)}
In many semantic theories, reference is treated as primitive, descriptivist, or truth-conditional.  Our work does not attempt to refute those approaches.  It proposes a different, explicit \emph{model}: reference emerges from minimizing an intrinsic mismatch penalty under a compression constraint.
\begin{itemize}
    \item \textbf{Frege-style}: reference is primitive. Here it is derived \emph{within the model} from cost minimization.
    \item \textbf{Russell-style}: reference is mediated by descriptions. Here ``description'' is replaced by intrinsic cost comparison and a mismatch penalty.
    \item \textbf{Possible-worlds}: semantic value is a truth-condition across worlds. We do not assume modal structure; we assume scale maps and a penalty.
\end{itemize}

Our framework aligns with information-theoretic viewpoints that treat representation as efficient coding (e.g.\ Shannon \cite{shannon1948}), but differs in that distortion is fixed to the canonical mismatch penalty $\Jcost$ and the meaning rule is set-valued (argmin), making the induced decision geometry explicit.

\subsection{Limitations}

\begin{enumerate}
    \item \textbf{Ratio embedding}: Our framework requires configurations to embed into $\R_{>0}$ via a ratio map. Not all semantic domains naturally admit such embeddings.

    \item \textbf{Single penalty}: We work with the canonical mismatch penalty $\Jcost$. Alternative penalties may be appropriate in domains where under- and over-shooting are not symmetric.

    \item \textbf{Static analysis}: The theory is synchronic. Incorporating learning or time-evolution requires additional structure (e.g., dynamics for $\iota$ or for admissible reference classes).
\end{enumerate}

\subsection{Future Directions}

\begin{enumerate}
    \item \textbf{Broader admissible reference.} Classify reference structures beyond the ratio-induced form (Definition~\ref{def:admissible-ref-struct}) for which analogues of the stability and compositionality theorems remain true.
    \item \textbf{Multi-dimensional ratios.} Extend the decision-geometry and boundary descriptions to $\iota:C\to(\R_{>0})^d$ with non-separable penalties, and quantify how coupling between coordinates affects stability margins.
    \item \textbf{Learning the scale map.} Given data of successful/unsuccessful references, formulate and analyze estimation procedures for $\iota$ (and admissible reference parameters) that preserve the proved invariances.
\end{enumerate}

\section{Conclusion}

This section summarizes the contributions and limitations of the model and records a few directions for refinement and application within the axioms fixed above.

{We have developed a mathematical \emph{model} of reference grounded in cost minimization. The theorem-level contributions are internal to the stated axioms and hypotheses.}

We summarize the main points:

\begin{enumerate}
    \item \textbf{Reference as compression}: Symbols are low-cost encodings of high-cost objects.

    \item \textbf{Canonical mismatch geometry}: The canonical penalty $\Jcost(x) = \frac{1}{2}(x + x^{-1}) - 1$ yields explicit decision boundaries and stability regions for finite dictionaries (Theorem~\ref{thm:geom-boundaries}).

    \item \textbf{Universal backbone}: Near-balanced configurations provide a provable backbone window around balance under admissible reference (Theorem~\ref{thm:backbone}). Global descriptive reach is obtained by composing many such low-cost primitives (Section~\ref{sec:compositionality}).

    \item \textbf{Compositionality}: Reference structures compose via products and sequences.
\end{enumerate}

The framework connects a simple optimization semantics with explicit geometric and compositional structure. Any application to a specific empirical domain requires specifying an appropriate scale map $\iota$ and verifying that the admissibility assumptions reasonably match that domain.

\section*{Acknowledgments}

We briefly acknowledge contributions and feedback that improved the exposition.

{We thank colleagues and readers for helpful discussions and feedback on earlier drafts.}

\appendix

\section{Classical characterization of the mismatch penalty}\label{app:dalembert}

This appendix records a classical functional-equation characterization showing that the explicit mismatch penalty used in the paper is essentially forced (up to scale) by the stated axioms.

{We prove Proposition~\ref{prop:Jcost-characterization}. The underlying functional-equation step is classical; see, for example, Acz\'{e}l \cite{aczel1966} or Kuczma \cite{kuczma2009}. We include the argument here to keep the manuscript self-contained and to clarify that the mismatch penalty is not introduced as a new object.}

\begin{lemma}[Convexity implies continuity]\label{lem:convex-continuous}
Let $I\subset\R$ be an open interval and let $g:I\to\R$ be finite-valued and convex. Then $g$ is continuous on $I$.
(See, e.g., Rockafellar \cite[Thm.~10.1]{rockafellar1970}.)
\end{lemma}

\noindent We apply this standard convexity fact to the mismatch penalty to obtain the regularity needed for the functional-equation classification.

\begin{lemma}[Regularity for the log-transformed d'Alembert equation]\label{lem:dalembert-regularity}
Assume $\Jcost$ satisfies Definition~\ref{def:cost-axioms}. Define $C:(0,\infty)\to\R$ by $C(x):=1+\Jcost(x)$ and $f:\R\to\R$ by $f(u):=C(e^{u})$. Then $f$ is continuous, satisfies
\[
  f(u+v)+f(u-v)=2f(u)f(v)\qquad(u,v\in\R),
\]
and obeys $f(0)=1$. In particular, the hypotheses of Lemma~\ref{lem:dalembert-classification} (and of the classical theorems of Acz'{e}l and Kuczma) apply to $f$.
\end{lemma}

\begin{proof}
By strict convexity (Definition~\ref{def:cost-axioms}(3)), $\Jcost$ is convex and finite-valued on $(0,\infty)$, hence continuous by Lemma~\ref{lem:convex-continuous}; therefore $C=1+\Jcost$ and $f(u)=C(e^{u})$ are continuous. The multiplicative identity in Definition~\ref{def:cost-axioms}(4) is equivalent to \eqref{eq:mult-dalembert-C} for $C$, and substituting $x=e^{u}$, $y=e^{v}$ yields the displayed d'Alembert equation for $f$. Finally $f(0)=C(1)=1$ by normalization.
\end{proof}

\begin{lemma}[Continuous solutions of d'Alembert's equation]\label{lem:dalembert-classification}
Let $f:\R\to\R$ be continuous and satisfy
\[
  f(t+s)+f(t-s)=2f(t)f(s)\qquad(t,s\in\R),
\]
with $f(0)=1$. Then either $f\equiv 1$, or there exists $a>0$ such that $f(t)=\cos(at)$ for all $t\in\R$, or there exists $a>0$ such that $f(t)=\cosh(at)$ for all $t\in\R$.
\end{lemma}

\begin{proof}
This classification is classical; see Acz\'{e}l \cite[Ch.~2]{aczel1966} or Kuczma \cite[Ch.~13]{kuczma2009}.
\end{proof}

\begin{proof}[Proof of Proposition~\ref{prop:Jcost-characterization}]
Let $\Jcost$ satisfy Definition~\ref{def:cost-axioms}. Define
\[
  C(x):=1+\Jcost(x)\qquad(x>0).
\]
Then \eqref{eq:dalembert} is equivalent to the multiplicative identity
\begin{equation}\label{eq:mult-dalembert-C}
  C(xy)+C(x/y)=2C(x)C(y)\qquad(x,y>0).
\end{equation}
{Define $f:\R\to\R$ by $f(t):=C(e^{t})$. By Lemma~\ref{lem:dalembert-regularity}, $f$ is continuous. Substituting $x=e^{t}$ and $y=e^{s}$ into \eqref{eq:mult-dalembert-C} yields d'Alembert's functional equation}
\begin{equation}\label{eq:dalembert-additive}
  f(t+s)+f(t-s)=2f(t)f(s)\qquad(t,s\in\mathbb R).
\end{equation}
{Moreover, $f(0)=C(1)=1$ and $f(t)\ge 1$ for all $t$ since $\Jcost\ge 0$.}

{By Lemma~\ref{lem:dalembert-classification}, the continuous solutions of \eqref{eq:dalembert-additive} with $f(0)=1$ are $f\equiv 1$, $f(t)=\cos(at)$, or $f(t)=\cosh(at)$ (for some $a>0$, with the constant solution corresponding to $a=0$).} The constraint $f(t)\ge 1$ rules out the cosine family unless $a=0$, and strict convexity rules out the constant solution. Hence there exists $a>0$ such that $f(t)=\cosh(at)$ for all $t$.

Undoing the change of variables gives
\[
  C(x)=f(\log x)=\cosh(a\log x),\qquad x>0,
\]
and therefore
\[
  \Jcost(x)=C(x)-1=\cosh(a\log x)-1=\tfrac12\big(x^{a}+x^{-a}\big)-1.
\]
Finally, note that
\[
\cosh(a\log(\iota_S/\iota_O)) -1 = \cosh\big(\log((\iota_S)^{a}/(\iota_O)^{a})\big)-1,
\]
so replacing $\iota_S,\iota_O$ by $\tilde\iota_S:=\iota_S^{a}$ and $\tilde\iota_O:=\iota_O^{a}$ absorbs the parameter $a$ into the scale maps and produces the normalized choice $a=1$ at the level of ratio-induced reference costs.
\end{proof}


\begin{thebibliography}{99}

\bibitem{frege1892}
G. Frege.
\newblock \"Uber Sinn und Bedeutung.
\newblock {\em Zeitschrift f\"ur Philosophie und philosophische Kritik}, 100:25--50, 1892.

\bibitem{russell1905}
B. Russell.
\newblock On denoting.
\newblock {\em Mind}, 14(56):479--493, 1905.

\bibitem{wigner1960}
E. Wigner.
\newblock The unreasonable effectiveness of mathematics in the natural sciences.
\newblock {\em Communications on Pure and Applied Mathematics}, 13(1):1--14, 1960.

\bibitem{harnad1990}
S. Harnad.
\newblock The symbol grounding problem.
\newblock {\em Physica D: Nonlinear Phenomena}, 42(1-3):335--346, 1990.

\bibitem{shannon1948}
C.E. Shannon.
\newblock A mathematical theory of communication.
\newblock {\em Bell System Technical Journal}, 27(3):379--423; 27(4):623--656, 1948.
\bibitem{kolmogorov1965}
A.N. Kolmogorov.
\newblock Three approaches to the quantitative definition of information.
\newblock {\em Problems of Information Transmission}, 1(1):1--7, 1965.

\bibitem{rissanen1978}
{J. Rissanen.
\newblock Modeling by shortest data description.
\newblock {\em Automatica}, 14(5):465--471, 1978.}

\bibitem{aczel1966}
J. Acz\'{e}l.
\newblock Lectures on Functional Equations and Their Applications.
\newblock Academic Press, 1966.

\bibitem{kuczma2009}
M. Kuczma.
\newblock An Introduction to the Theory of Functional Equations and Inequalities: Cauchy's Equation and Jensen's Inequality.
\newblock 2nd edition, Birkh\"{a}user, 2009.

\bibitem{rockafellar1970}
R.T. Rockafellar.
\newblock \emph{Convex Analysis}.
\newblock Princeton University Press, 1970.

\bibitem{yablo2014aboutness}
{S.\ Yablo.
\newblock \emph{Aboutness}.
\newblock Princeton University Press, 2014.}

\bibitem{hawke2018theories}
{P.\ Hawke.
\newblock Theories of aboutness.
\newblock \emph{Australasian Journal of Philosophy}, 96(4):697--723, 2018.
\newblock doi:10.1080/00048402.2017.1388826.}

\bibitem{rothschild2017yablo}
{D.\ Rothschild.
\newblock Yablo's semantic machinery.
\newblock \emph{Philosophical Studies}, 174(3):787--796, 2017.
\newblock doi:10.1007/s11098-016-0759-3.}

\bibitem{fine2020yablo}
{K.\ Fine.
\newblock Yablo on subject-matter.
\newblock \emph{Philosophical Studies}, 177(1):129--171, 2020.
\newblock doi:10.1007/s11098-018-1183-7.}

\bibitem{yablo2018reply}
{S.\ Yablo.
\newblock Reply to Fine on aboutness.
\newblock \emph{Philosophical Studies}, 175(6):1495--1512, 2018.
\newblock doi:10.1007/s11098-017-0922-5.}

\bibitem{fine2014truthmaker}
{K.\ Fine.
\newblock Truth-maker semantics for intuitionistic logic.
\newblock \emph{Journal of Philosophical Logic}, 43(2--3):549--577, 2014.
\newblock doi:10.1007/s10992-013-9281-7.}

\bibitem{washburn2026RG}
{Washburn, J.; Zlatanovi\'c, M.; Allahyarov, E.
\newblock \emph{Recognition Geometry}.
\newblock \emph{Axioms}, to appear (accepted, 2026).}

\end{thebibliography}
\end{document}